\documentclass[a4paper,12pt]{amsart}

\usepackage{amssymb}
\usepackage[normalem]{ulem}
\usepackage{color}

\newcommand\field[1]{\mathbb{#1}}

\newcommand\ZZ{\field{Z}}
\newcommand\Bb{\mathcal B}

\newcommand\Tt{\mathcal T}

\newcommand\go{G^{(0)}}

\newcommand\interior{\operatorname{Int}}

\theoremstyle{plain}
\newtheorem{theorem}{Theorem}[section]
\newtheorem*{theorem*}{Theorem}
\newtheorem*{prop*}{Proposition}

\newtheorem{lemma}[theorem]{Lemma}

\theoremstyle{remark}
\newtheorem{rmk}[theorem]{Remark}

\newtheorem{example}[theorem]{Example}

\newtheorem*{ack}{Acknowledgements}

\theoremstyle{definition}

\newtheorem{dfn}[theorem]{Definition}

\newcommand{\iso}{\operatorname{Iso}}

\numberwithin{equation}{section}

\title{A groupoid formulation of the Baire Category Theorem}

\author{Jonathan Brown}
\author{Lisa Orloff Clark}

\address{Lisa Orloff Clark\\
Department of Mathematics and Statistics\\ 
University of Otago\\
 PO Box 56\\ Dunedin
Dunedin 9054\\
New Zealand.}
\email{lclark@maths.otago.ac.nz}

\address{Jonathan Brown\\
    Mathematics Department\\
    Kansas State University\\
138 Cardwell Hall\\
Manhattan, KS 66506-2602\\
USA.}
\email{brownjh@math.kansas.edu}

\date{September 12, 2012}

\begin{document}



\begin{abstract}
 We prove that the Baire Category Theorem  is equivalent to the following: 
Let $G$ be a topological groupoid such that the unit space is a complete metric space, and 
\item there is a countable cover of $G$ by neighbourhood bisections.
If $G$ is effective, then $G$ is topologically principal.
\end{abstract}

\maketitle

\section{Introduction}

\begin{theorem}[The Baire Category Theorem] (See, for example,  \cite[Theorem~7.7.2]{M}.)
\label{BCT}
 Suppose  $X$ is a complete metric space.  If $\{C_n\}$ is a countable collection 
of closed subsets of $X$, each with empty interior, then $\bigcup_n C_n$ has empty interior. 
\end{theorem}

The proof of the Baire Category Theorem, originally formulated by Baire in the 1890's,
 requires a variant of the Axiom of Choice  \cite[Chapter 13]{D}.
In fact, \cite{Blair} and \cite{Goldblatt} show the Baire Category Theorem is equivalent to the \emph{the Principle of Dependant Choices} 
which says:
\begin{quote}
Suppose $X$ is a set and  $R \subseteq X \times X$ is a relation  such that for each $x \in X$, 
there exists $y \in X$ such that 
$(x,y) \in R$.  Then there is a sequence $\{x_n\} \subseteq X$ such that $(x_n,x_{n+1}) \in R$ for all $n$. 
\end{quote}
The Principle of Dependent Choices falls strictly between the 
Countable Axiom of Choice and
the Axiom of Choice; see \cite{jech} for more details. 

In this note, we show the Baire Category Theorem is equivalent to Theorem~\ref{GET} below, which is a surprising result 
about \emph{effective groupoids}.  
We discovered a version of Theorem~\ref{GET} in our 
study of simple groupoid $C^*$-algebras \cite[Lemma~3.3]{BCFS}.  A version also appears in \cite[Proposition~3.6]{Renault:IMSB08} in the 
context of
maximal abelian subalgebras of $C^*$-algebras.  That Theorem~\ref{GET} implies the Baire Category Theorem is entirely new; but even our proof 
that the Baire Category Theorem implies Theorem~\ref{GET} is different from  those in \cite{BCFS} and \cite{Renault:IMSB08}.

\section{Preliminaries} 
A groupoid $G$ is a small category in which every morphism is invertible.  
We identify the set of objects of $G$ with the set of identity morphisms and denote this set 
$\go$. For $\gamma\in G$, we denote the range and source (domain) of $\gamma$ by $r(\gamma)$ and $s(\gamma)$ respectively.  
Thus $r,s: G\to \go\subseteq G$.  We define $G^{(2)}:=\{(\gamma,\eta)\in G\times G: r(\eta)=s(\gamma)\}$; $G^{(2)}$ 
consists of precisely those pairs of morphisms that can be composed in $G$.
For any $x\in \go$, the \emph{isotropy group} at $x$ is the group \[xGx:=\{\gamma\in G: r(\gamma)=s(\gamma)=x\}.\]  
The \emph{isotropy subgroupoid} of $G$ is  
\[\operatorname{Iso}(G):=\bigcup_{x\in \go} xGx\]
which is itself a groupoid.   If $B \subseteq G$, then we also write $\iso(B) := \iso(G) \cap B$.
We say $G$ is a \emph{group bundle} if $\operatorname{Iso}(G)=G$.

A groupoid $G$ is  a \emph{topological groupoid} if $G$ is equipped with a topology so that composition 
and inversion are continuous.  In this case, $r$ and $s$ are continuous maps. If $\go$ is Hausdorff, then the continuity of 
$r$ and $s$ implies that $\iso(G)$ is a closed subset of $G$.

An open set $A\subseteq G$ is called  an \emph{open bisection} if $r(A)$ and $s(A)$ are open in $G$ and
 $r$ and $s$ restricted
to $A$ are homeomorphisms onto their image; in particular $r$ and $s$ are injective on $A$.  

We say a groupoid $G$ is \emph{topologically principal} if the set $\{x\in \go: xGx\ \neq \{x\}\}$  has empty interior in $\go$.
A groupoid $G$ is \emph{effective} if 
$\operatorname{Iso}(G) - \go$ has empty interior.  

\section{When does effective imply topologically principal?}
In \cite[Proposition~3.6(ii)]{Renault:IMSB08} Renault considers effective groupoids whose unit spaces are `Baire'.  
We can interpret Renault's result as saying that Theorem~\ref{BCT}  implies the following:
\begin{theorem} \cite[Proposition~3.6(ii)]{Renault:IMSB08}
\label{thm: effect ren}
Suppose  $G$ is a topological groupoid such that:
 \begin{enumerate}
 \item the unit space is a complete metric space; 
 \item \label{it:2ren} $G$ has a countable cover consisting  of open bisections.
 \end{enumerate}
 If $G$ is effective, then $G$ is topologically principal.  
 \end{theorem}

Our original intention was to show that Theorem~\ref{thm: effect ren} is equivalent to  Theorem~\ref{BCT}.  However, we
eventually realised that such a result will only hold if we weaken the hypotheses of Theorem~\ref{thm: effect ren}.   
To see why, consider the class of effective groupoids constructed in Example~\ref{ex: main} below.   Each has the property that 
Theorem~\ref{BCT} implies it is topologically principal.  At the same time,  groupoids in this class  may not satisfy the hypotheses of 
Theorem~\ref{thm: effect ren}.  (We will also use this class of examples later in the proof of our main result.)

\begin{example}
\label{ex: main}
Let $X$ be a complete metric space and $\{C_n\}$ be a countable
collection of closed subsets of $X$, each with empty interior.  Define $C:=\bigcup_n C_n$.
 
Let $G$ be the group bundle with unit space $X$ and isotropy groups 
\begin{equation*}xGx := 
\begin{cases} \mathbb{Z}_2, & \text{ if $x \in C$;}\\
            \{ 0 \}, & \text{otherwise.}
\end{cases}
\end{equation*}
We identify the identity element $0 \in xGx$ with $x$.  
For each $x \in C$, we write $\gamma_x$ for the nontrivial element of $xGx$.   Notice that
\[G^{(2)} = \{(x,x) : x\in X\} \cup \{(\gamma_x, \gamma_x): x \in C\} \cup  \{(x, \gamma_x): x \in C\} \cup  \{(\gamma_x, x): x \in C\}.\]

To make $G$ into a topological groupoid, first let  $\Tt$ be the topology for $X$.  Define the collection
\[
 \Bb := \Tt \cup  \{V \subseteq G : V = (W - \{ x \}) \cup \{ \gamma_x\} \text{ for some } W \in \Tt \text{ and  } x \in C\cap W\}.
\]
The collection $\Bb$ forms a basis for a topology on $G$.  To see this, note that $\Bb$ covers $G$ and 
since $\Tt$ is the topology for $X$, it is easy to see that $U, V\in \Bb$ implies $U\cap V\in \Bb$.
We claim that $G$ endowed with the topology generated by $\Bb$ is a topological groupoid. Indeed, inversion
 is given by the identity and is thus continuous.
Now let $m: G^{(2)}\to G$ be the composition map.  Fix $V \in \Bb$.  If $V \in \Tt$, then
\[m^{-1}(V) = \{(x,x):x \in V\} \cup \{(\gamma_x, \gamma_x): x \in V\}.\]
If $V = (W - \{ y \}) \cup \{ \gamma_y\}$ for some $W \in \Tt$ and $y \in C \cap W$, then
\[m^{-1}(V) = \{(x,x):x \in W - \{y\} \} \cup \{(\gamma_y, y), (y, \gamma_y)\}\cup \{(\gamma_x,\gamma_x): x\in (W-\{y\})\cap C\} .\]
In both cases, it is straightforward to show that $m^{-1}(V)$ is open in $G^{(2)}$, hence composition is continuous
as claimed.  
 
Since every element of $\Bb$ intersects the unit space, 
the set $\operatorname{Iso}(G) - \go= G - \go$ contains no open sets, so $G$ is effective.   
By construction, $\go = X$ is a complete metric space and
\[
 C = \{x\in \go: xGx\ \neq \{x\}\},
\] 
so  Theorem~\ref{BCT} implies that $G$ is also topologically principal.
Notice that $G$ need not satisfy item~\eqref{it:2ren} of Renault's Theorem~\ref{thm: effect ren}.  
Indeed if $X=[0,1]$ and $C_n=C$ is the Cantor set for all $n$, there is no countable cover of 
$G$ consisting of open bisections.  To see this suppose $\{U_i\}$ is any countable open cover of 
$G$.  Since the Cantor set is uncountable there exists an $i_0$ such that 
$A_{i_0}:=\{x\in C: \gamma_x\in U_{i_0}\}$ is uncountable.  For each 
$x\in A_{i_0}$ pick a basis element $(V_x-\{x\})\cup \{\gamma_x\}$ contained in $U_{i_0}$.  
Since the standard basis for $[0,1]$ is given by connected intervals, we can assume that $V_x$ is connected.
For each $n \in \ZZ^+$ define $D_n:=\{x\in A_{i_0}: \text{~diameter of~} V_x\text{~is greater than~} 2/n\}$.  
Since $A_{i_0}$ is uncountable, there exists $n_0$ such that $D_{n_0}$ is uncountable.  Now consider 
the partition $\{P_m:=[m/2n_0, (m+1)/2n_0]\}$ of $[0,1]$ where $0 \leq m \leq 2n_0-1$.   
Since $D_{n_0}$ is uncountable, there 
exists an $m_0$ such that $D_{n_0}\cap P_{m_0}$ is uncountable. By the definition of $D_{n_0}$ this implies 
that for every $x\in D_{n_0}\cap P_{m_0}$ both $x$ and $\gamma_x$ are in $U_{i_0}$ and so $U_{i_0}$ 
is \emph{not} an open bisection.
\end{example}

While the groupoids considered above need not have a countable cover of 
open bisections, they do have a countable cover consisting
of `well-behaved sets'.  
We call these sets \emph{neighbourhood bisections}.  
(We denote the interior of a set $D$ by $\interior(D)$.)

\begin{dfn}
\label{def:nb}
 A set $B \subseteq G$ is called a \emph{neighbourhood bisection} 
if the following hold:
\begin{enumerate}
\item \label{it:nb_1} $B \subseteq \overline{\interior (B)}$;
\item \label{it:nb_2}$r|_B$ and $s|_B$ are injective;
\item \label{it:nb_3}$r(B)$ and $s(B)$ are open in $G$;
\item \label{it:nb_4}$\interior (B)$ is an open bisection;
\item \label{it:nb_5}$B - \interior(B) \subseteq \iso(B) - \go$.
\end{enumerate}
\end{dfn}

In the next section we prove the following theorem is equivalent to Theorem~\ref{BCT}.  One part of 
our proof involves showing  the class of groupoids constructed in Example~\ref{ex: main} do indeed
have a countable cover consisting of neighbourhood bisections. 

\begin{theorem}
 \label{GET}
Suppose  $G$ is a topological groupoid such that:  
\begin{enumerate}
\item \label{it:GET_1} the unit space is a complete metric space; 
\item \label{it:GET_2}  $G$ has a countable cover consisting  of neighbourhood bisections.
\end{enumerate}
If $G$ is effective, then $G$ is topologically principal.
\end{theorem}

\begin{rmk}
Suppose $G$ is a groupoid satisfying the hypotheses of Theorem~\ref{GET}, then $\go$ is open in $G$.  
To see this, let $\{B_n\}$ be a countable cover of $G$ by neighbourhood bisections, then $\go=\bigcup r(B_n)$
 which is open.
\end{rmk}

\begin{rmk}
An \emph{\'etale} groupoid is a topological groupoid that has a cover consisting 
of open bisections.  
When studying $C^*$-algebras associated to groupoids, one often considers second-countable, 
locally compact, Hausdorff groupoids that are  \'etale.
These groupoids satisfy the hypothesis of Theorem~\ref{GET}.
\end{rmk}

\begin{rmk}
Suppose $G$ is a topological groupoid.  If $r$ is an open map, then $G$ topologically principal implies $G$ is effective.  
See \cite[Examples~6.3 and 6.4]{BCFS} for examples of
groupoids (that do not satisfy the hypothesis of  Theorem \ref{GET}) that are effective but not topologically principal.
\end{rmk}
\section{Main Result}

\begin{theorem}
 \label{main}
Theorem~\ref{BCT} is equivalent to Theorem~\ref{GET}.
\end{theorem}

Before we prove Theorem~\ref{main}, we establish the following two lemmas.  The first
 lemma is used to prove the second (Lemma~\ref{lem:main}); Lemma~\ref{lem:main} is a key step in our proof of Theorem~\ref{main}.

\begin{lemma}
\label{lem:1}
Suppose $G$ is a topological groupoid such that $\go$ is open in $G$, $B \subseteq G$ 
is a neighbourhood bisection and $D \subseteq B$  is closed in $B$ where $B$ is endowed with the subspace topology.  
Suppose that
$B - \interior(B) \subseteq D$.  Then
$r(D)$ is closed in $r(B)$ where $r(B)$ is endowed with the subspace topology. 
\end{lemma}

\begin{proof}
Let $G$, $B$ and $D$ be as stated.  Then 
\[D = (D \cap \interior(B) ) \cup (B - \interior(B))\]
which means 
\[r(D) = r(D \cap \interior(B) ) \cup r(B - \interior(B)).\]
Since $r|_B$ is a bijection onto its image, $r(B - \interior(B)) = r(B) - r(\interior(B))$ which is closed in 
$r(B)$ as  $r(\interior(B))$ is open. Further $r(D \cap \interior(B))$ is closed in 
$r(\interior(B))$ because $r$ restricted to $\interior(B)$ is a 
homeomorphism.    Thus there exists a closed set $C$  such that 
$r(\interior(B)) \cap C = r(\interior(B) \cap D)$. Therefore 
\begin{align*}
r(D) &= r(\interior(B) \cap D) \cup r(B - \interior(B))\\
&= (r(\interior(B)) \cap C) \cup (r(B - \interior(B)) \cap C) \cup r(B - \interior(B))\\
&= (r(B) \cap C) \cup r(B - \interior(B))
\end{align*}
which is closed in $r(B)$. 
                                                             
\end{proof}
\break
\begin{lemma}
\label{lem:main}
 Suppose $G$ is an effective groupoid such that $\go$ is open in $G$ and $B$ is a 
neighbourhood bisection. Then
 \begin{enumerate}
 \item \label{it: empty nbhd}  $r(\iso(B) - \go)$ has empty interior, and
 \item  \label{it: empty cls} $\overline{r(\iso(B) - \go)}$  has empty interior.
 \end{enumerate}
\end{lemma}

\begin{proof}

For \eqref{it: empty nbhd}, by way of contradiction, suppose
there exists  a nonempty open  set $W \subseteq r(\iso(B) - \go)$.  
Thus $W \cap r(B) \neq \emptyset$, and since $ B \subseteq \overline{\interior(B)}$, we have $ 
W \cap r\left(\overline{\interior(B)}\right) \neq \emptyset$.  Therefore
\[ W \cap \overline{r(\interior(B))} \neq \emptyset 
\text{  because  $r\left(\overline{\interior(B)}\right) \subseteq \overline{r(\interior(B))}$.} \]
Hence $W \cap r(\interior(B))$ is a nonempty open set contained in $\go$.  Since 
\[
\phi:=r|_{\interior(B)}
\] is a homeomorphism, 
\[\phi^{-1}(W \cap r(\interior(B)))\]
is a nonempty open subset of $\interior(B)$ and thus is open in $G$.
Since  $r$ is injective on $B$ and $W \subseteq r(\iso(B) - \go)$,
\[(\phi^{-1}(W \cap r(\interior(B))) \subseteq \iso(B) - \go \subseteq \iso(G) - \go.\]
This is a contradiction because $G$ is effective.

For \eqref{it: empty cls}, by way of contradiction, assume there exists a nonempty open subset 
\[V \subseteq  \overline{r(\iso(B) - \go)}.\] Notice
that $V \cap r(B)$ is a nonempty open subset of $\go$.
Further,
\[
V \cap r(B) \subseteq \overline{r(\iso(B) - \go)} \cap r(B).
\]
 
We show $\overline{r(\iso(B) - \go)} \cap r(B)=r(\iso(B) - \go).$
Since $\iso(B)$ is closed in $B$ and $\go$ is open,  $\iso(B) - \go$ is also closed in $B$.  Also,
$B - \interior(B) \subseteq \iso(B) - \go$ by assumption.  Therefore we apply Lemma~\ref{lem:1} to see that
$r(\iso(B) - \go)$ is closed in $r(B)$.  
Thus 
\begin{align*}
r(\iso(B) - \go) &= \overline{r(\iso(B) - \go)} \cap r(B), \text{ and so }\\
V \cap r(B) &\subseteq r(\iso(B) - \go)
\end{align*}
which contradicts item \eqref{it: empty nbhd}.  
\end{proof}

\begin{proof}[Proof of Theorem~\ref{main}:]
Suppose Theorem~\ref{BCT} holds.  Let $G$ be a topological groupoid with a countable 
cover of neighbourhood bisections  $\{ B_n\}$  such that $G^{(0)}$ is a complete metric space.
Suppose also that $G$ is effective.

By Lemma~\ref{lem:main}\eqref{it: empty cls}, the set $\overline{r(\iso(B_n)-\go)}$ 
has empty interior for every $n$.    Define $C_n:= \overline{r(\iso(B_n)-\go)} \cap \go$ for each $n$.  
Notice
that each $C_n$ is a closed 
subset of $\go$.  Because $\go$ is open in $G$,  each $C_n$ also has empty interior in $\go$.  
Applying Theorem~\ref{BCT} (Baire Category Theorem) to the collection $\{C_n\}$ we see that
\[C:= \bigcup_n C_n
 \]
 has empty interior.  By construction, 
$C$ contains the units with nontrivial isotropy.  Therefore, $G$ is topologically principal.

Conversely, suppose that Theorem~\ref{GET} holds.  Let $X$ be a complete metric space with topology $\Tt$ 
and $\{C_n\}$ be a countable
collection of closed subsets of $X$, each with empty interior. With out loss of generality we can assume $C_0=\emptyset$.  Let $C=\bigcup_n C_n$.  
Define $G$ as in Example~\ref{ex: main}.  Since $\go=X$ as a topological space, $G$ satisfies \eqref{it:GET_1} of Theorem~\ref{GET}.  

For each $n$, define 
\[B_n:=\left(X-C_n\right) \cup \{\gamma_x : x\in C_n\}.  \]
We claim that each $B_n$ is a neighbourhood bisection.  To prove this, we must check
each of the items in Definition~\ref{def:nb}.    
To see \eqref{it:nb_1}, first note that $X- C_n$ is open in $G$ and contained in $B_n$.  
Thus $X- C_n\subseteq \interior(B_n)$.  We show that $B_n\subseteq \overline{X- C_n}\subseteq \overline{\interior(B_n)}$.    
Consider $\gamma_x$ for some $x\in C_n$.  For every $V\in \Bb$ with $\gamma_x\in V$ we have 
$V=W-\{x\}\cup \{\gamma_x\}$ where $W\in \Tt$ and $x\in W$.  
Now $V\cap (X- C_n)=W\cap (X- C_n)$ is nonempty because $C_n$ has empty interior.  
Therefore $\gamma_x\in \overline{X- C_n}$.  Since $B_n=X- C_n\cup \{\gamma_x: x\in C_n\}$, 
we have $B_n\subseteq \overline{X- C_n}$.  
That $B_n$ satisfies item~\eqref{it:nb_2} is clear and  $r(B_n)=X=s(B_n)$ is open in $G$ 
giving us item~\eqref{it:nb_3}.
Since $r(V)$ is  in $\Tt$ for every $V \in \Bb$,  $r=s$ is an open map and hence  $r|_{\interior(B_n)}=s|_{\interior(B_n)}$ 
is a homeomorphism with open image giving item~\eqref{it:nb_4}. 
Lastly, since $B_n \cap \go = X- C_n \subseteq \interior(B_n)$ and $G = \iso(G)$, we 
get item~\eqref{it:nb_5}. Thus $\{B_n\}_n$
is a countable cover of $G$ by neighbourhood bisections and $G$ satisfies
item (\ref{it:GET_2}) of Theorem~\ref{GET}.   

We showed in Example~\ref{ex: main} that $G$ is effective.  Therefore 
$G$ is topologically principal by Theorem~\ref{GET}.   Thus 
\[
 C = \{x\in \go: xGx\ \neq \{x\}\}
\]
has empty interior proving Theorem~\ref{BCT}.
\end{proof}

\begin{ack} Thanks to Fran\c{c}ois Dorais for the very helpful email correspondence.\end{ack}



\begin{thebibliography}{99}

\bibitem{Blair}C.E. Blair, \emph{The Baire Category Theorem implies the Principle of Dependent Choices}, 
Bull. Acad. Polon. Sci. Sér. Sci. Math. Astronom. Phys. \textbf{25} (1977), 933--934.

\bibitem{BCFS} J.H. Brown, L.O. Clark, C. Farthing and A. Sims, \emph{Simplicity of algebras associated to \'etale groupoids}, 
submitted. arXiv:1204.3127v1 [math.OA]

\bibitem{D}W. Dunham, \emph{The calculus gallery:
masterpieces from Newton to Lebesgue}, Princeton University Press, Princeton, NJ, 2005.

\bibitem{Goldblatt}R. Goldblatt, \emph{On the role of the Baire Category Theorem and 
Dependent Choice in the foundations of logic}, J. Symbolic Logic \textbf{50} (1985), 412--422.  

\bibitem{jech}T.J. Jech, \emph{The Axiom of Choice}, reprint, Dover Publications, Inc., Mineola, NY, 2008.

\bibitem{M} J.R. Munkres, \emph{Topology}, second edition, Prentice Hall, Upper Saddle River, 2000.

\bibitem{Pat} A. Paterson,\emph{ Groupoids, inverse semigroups, and their operator algebras,}  Birkh\"{a}user Boston, Inc., Boston, MA, 1999.

\bibitem{Renault:IMSB08} J. Renault, \emph{Cartan subalgebras in {$C^*$}-algebras}, Irish Math.
    Soc. Bulletin \textbf{61} (2008), 29--63.


\end{thebibliography}
\end{document}